\documentclass[3p]{article}

\usepackage{amsfonts}
\usepackage{amsthm}
\usepackage{amsmath}
\usepackage{amssymb}
\usepackage{mathtools} 
\usepackage{bbm} 

\usepackage[top=1in, bottom=1in, left=1in, right=1in]{geometry}
\usepackage{xcolor}

\usepackage{graphicx}
\usepackage{epstopdf}
\usepackage{subfigure}

\usepackage{authblk}

\theoremstyle{plain}
\newtheorem{theorem}{Theorem}[section]
\newtheorem{corollary}[theorem]{Corollary}
\newtheorem{lemma}[theorem]{Lemma}

\theoremstyle{definition}
\newtheorem{definition}[theorem]{Definition}
\newtheorem{example}[theorem]{Example}
\newtheorem{question}{Question}[section]

\theoremstyle{remark}
\newtheorem{remark}[theorem]{Remark}

\DeclareSymbolFont{pxfontssymbolsC}{U}{pxsyc}{m}{n}
\DeclareMathSymbol{\coloneqq}{\mathrel}{pxfontssymbolsC}{66}

\begin{document}



\title{On fractional calculus with general analytic kernels}

\date{}

\author[1,2]{Arran Fernandez\thanks{Corresponding author. Email: \texttt{arran.fernandez@emu.edu.tr}}}
\author[2]{Mehmet Ali \"Ozarslan\thanks{Email: \texttt{mehmetali.ozarslan@emu.edu.tr}}}
\author[3,4]{Dumitru Baleanu\thanks{Email: \texttt{dumitru@cankaya.edu.tr}}}

\affil[1]{{\small Department of Applied Mathematics and Theoretical Physics, University of Cambridge, Wilberforce Road, CB3 0WA, United Kingdom}}
\affil[2]{{\small Department of Mathematics, Faculty of Arts and Sciences, Eastern Mediterranean University, Famagusta, Northern Cyprus, via Mersin-10, Turkey}}
\affil[3]{{\small Department of Mathematics, Cankaya University, 06530 Balgat, Ankara, Turkey}}
\affil[4]{{\small Institute of Space Sciences, Magurele-Bucharest, Romania}}

\maketitle

\begin{abstract}
Many possible definitions have been proposed for fractional derivatives and integrals, starting from the classical Riemann--Liouville formula and its generalisations and modifying it by replacing the power function kernel with other kernel functions. We demonstrate, under some assumptions, how all of these modifications can be considered as special cases of a single, unifying, model of fractional calculus. We provide a fundamental connection with classical fractional calculus by writing these general fractional operators in terms of the original Riemann--Liouville fractional integral operator. We also consider inversion properties of the new operators, prove analogues of the Leibniz and chain rules in this model of fractional calculus, and solve some fractional differential equations using the new operators.
\end{abstract}



\section{Background} \label{sec:intro}

In fractional calculus, we seek to extend the basic calculus operators of differentiation and integration, generalising the order of these operators beyond the integers to the real line or the complex plane. The question of how to define, for example, the $\frac{1}{2}$th derivative of a function is one that has intrigued mathematicians and scientists for hundreds of years \cite{miller-ross,oldham-spanier}. Even today there is no single unique answer to this fundamental question, but many different definitions of fractional calculus have been proposed, starting from various viewpoints, and each one has its own advantages and disadvantages \cite{samko-kilbas-marichev,hristov2}.

One of the most natural and popular models of fractional calculus is the \textbf{Riemann--Liouville} one \cite{miller-ross,oldham-spanier}. Here, the $\alpha$th fractional integral of a function $f$, with constant of integration $a$, is defined by
\begin{equation}
\label{RLdef:int}
\prescript{RL}{}I_{a+}^{\alpha}f(t)\coloneqq\frac{1}{\Gamma(\alpha)}\int_a^t(t-\tau)^{\alpha-1}f(\tau)\,\mathrm{d}\tau,\quad\mathrm{Re}(\alpha)>0,
\end{equation}
while the $\alpha$th fractional derivative of $f$, again depending on a constant $a$, is defined by
\begin{equation}
\label{RLdef:deriv}
\prescript{RL}{}D_{a+}^{\alpha}f(t)\coloneqq\frac{\mathrm{d}^m}{\mathrm{d}t^m}\Big(\prescript{RL}{}I_{a+}^{m-\alpha}f(t)\Big),\quad\mathrm{Re}(\alpha)\geq0,m\coloneqq\lfloor\mathrm{Re}(\alpha)\rfloor+1.
\end{equation}
The term ``differintegration'' is used to cover both integration and differentiation, which are now distinguished only by the sign of the real part of the order. The Riemann--Liouville model has discovered applications in many areas of science -- see for example \cite{bagley,hilfer,kilbas-srivastava-trujillo,magin,petras,podlubny} and the references therein.

The definition \eqref{RLdef:deriv} of Riemann--Liouville fractional derivatives can be modified by interchanging the operations of differentiation and fractional integration. This gives rise to the \textbf{Caputo} model \cite{caputo}. Here, fractional integrals are taken in the Riemann--Liouville sense \eqref{RLdef:int}, while the $\alpha$th fractional derivative of a function $f$, with constant of differintegration $a$, is defined by
\begin{equation}
\label{CAPdef:deriv}
\prescript{C}{}D_{a+}^{\alpha}f(t)\coloneqq\prescript{RL}{}I_{a+}^{m-\alpha}\Big(\frac{\mathrm{d}^m}{\mathrm{d}t^m}f(x)\Big),\quad\mathrm{Re}(\alpha)\geq0,m\coloneqq\lfloor\mathrm{Re}(\alpha)\rfloor+1.
\end{equation}
The Caputo definition is often preferred for modelling initial value problems, although analyticity in the order of differintegration is lost \cite{podlubny,samko-kilbas-marichev}.

One of the most simple and efficient models of fractional calculus is due to Liouville \cite{liouville}. It was originally obtained for a particular set of functions, but it can encapsulate both Riemann--Liouyille and Caputo derivatives as special cases.

In recent years, researchers have proposed new fractional models based on replacing the power function kernel in \eqref{RLdef:int} by a different, singular or non-singular, kernel function. The motivation behind such proposals relates to the various real data corresponding to different complex systems requiring different kernels.


For example, the \textbf{Atangana--Baleanu} (or AB) fractional model \cite{atangana-baleanu}, proposed in 2016, is based on replacing the power function kernel of \eqref{RLdef:int} by a non-singular function known to have strong connections with fractional calculus \cite{haubold-mathai-saxena,mathai-haubold}, namely the Mittag-Leffler function. In this model, there are two ways of defining the $\alpha$th fractional derivative of a function $f$ with constant of differintegration $a$; these are referred to as the AB derivatives of Riemann--Liouville and Caputo type respectively, by comparison with \eqref{RLdef:deriv} and \eqref{CAPdef:deriv}:
\begin{alignat}{2}
\label{ABRdef:deriv}
\prescript{ABR}{}D_{a+}^{\alpha}f(t)&\coloneqq\frac{B(\alpha)}{1-\alpha}\frac{\mathrm{d}}{\mathrm{d}t}\int_a^tE_{\alpha}\left(\frac{-\alpha}{1-\alpha}(t-\tau)^{\alpha}\right)f(\tau)\,\mathrm{d}\tau,&&\quad0<\alpha<1, \\
\label{ABCdef:deriv}
\prescript{ABC}{}D_{a+}^{\alpha}f(t)&\coloneqq\frac{B(\alpha)}{1-\alpha}\int_a^tE_{\alpha}\left(\frac{-\alpha}{1-\alpha}(t-\tau)^{\alpha}\right)f'(\tau)\,\mathrm{d}\tau,&&\quad0<\alpha<1,
\end{alignat}
where the function $B$ satisfies $B(0)=B(1)=1$ and is often \cite{baleanu-fernandez} taken to be real and positive.

Another recently proposed model of fractional calculus, called the \textbf{generalised proportional fractional} (or GPF) model, is based on the following fractional integral operator with two parameters:
\begin{equation}
\label{GPFdef:int}
\prescript{GPF}{}I_{a+}^{\alpha,\rho}f(t)\coloneqq\frac{1}{\rho^{\alpha}\Gamma(\alpha)}\int_a^t\exp\left(\frac{\rho-1}{\rho}(t-\tau)\right)(t-\tau)^{\alpha-1}f(\tau)\,\mathrm{d}\tau,\quad0<\rho\leq1,\mathrm{Re}(\alpha)>0,
\end{equation}
This operator and its derivative were analysed in detail in \cite{jarad-abdeljawad-alzabut}.

The above formulae can be viewed as special cases of the \textbf{Prabhakar} fractional model, which was introduced in 1971 \cite{prabhakar} for solving an integral equation, and later \cite{kilbas-saigo-saxena} interpreted as a fractional differintegral. Here, the fractional integral of a function $f$ with constant of differintegration $c$, with parameters $\alpha,\beta,\omega,\rho$ determining the order, is defined by
\begin{equation}
\label{PRABdef:int}
\prescript{P}{\beta,\omega}I_{a+}^{\alpha,\rho}f(t)\coloneqq\int_a^t(t-\tau)^{\alpha-1}E_{\beta,\alpha}^{\rho}\left[\omega(t-\tau)^{\beta}\right]f(\tau)\,\mathrm{d}\tau,\quad\mathrm{Re}(\alpha)>0,\mathrm{Re}(\beta)>0,
\end{equation}
while the fractional derivative of $f$ with the same parameters is defined similarly to \eqref{RLdef:deriv}, namely by
\begin{equation}
\label{PRABdef:deriv}
\prescript{P}{\beta,\omega}D_{a+}^{\alpha,\rho}f(t)\coloneqq\frac{\mathrm{d}^m}{\mathrm{d}t^m}\left(\prescript{P}{\beta,\omega}I_{a+}^{m-\alpha,-\rho}f(t)\right),\quad\mathrm{Re}(\alpha)>0,\mathrm{Re}(\beta)>0,m\coloneqq\lfloor\mathrm{Re}(\alpha)\rfloor+1.
\end{equation}
The function $E_{\beta,\alpha}^{\rho}$ appearing in \eqref{PRABdef:int} is a generalisation of the Mittag-Leffler function defined by
\begin{equation}
\label{GMLdef}
E_{\beta,\alpha}^{\rho}(x)\coloneqq\sum_{n=0}^{\infty}\frac{\Gamma(\rho+n)}{\Gamma(\rho)\Gamma(\beta n+\alpha)n!}x^n.
\end{equation}
Although it is older than some of the other models mentioned above, the Prabhakar formula has gone mostly unnoticed until recent years. But now it has begun to attract attention, and applications have been discovered e.g. in viscoelasticity \cite{giusti-colombaro} and stochastic processes \cite{gorenflo-kilbas-mainardi-rogosin}.

The above covers only a few of the many definitions which have been proposed for fractional derivatives and integrals. There are other definitions dating back a hundred years or more \cite{samko-kilbas-marichev}, and more recently various generalisations of some of the above models have been proposed \cite{fernandez-baleanu,garra-gorenflo-polito-tomovski,ozarslan-ustaoglu1,ozarslan-ustaoglu2,srivastava-tomovski}.

The wealth of different definitions available for fractional derivatives and integrals has caused some researchers to propose criteria to determine whether a particular operator should be called a fractional derivative or not. But despite several attempts \cite{ross,ortigueira-machado,tarasov,caputo-fabrizio2}, so far there is no universally accepted set of criteria for this.

\begin{question}
Is it possible to define a \textit{general} class of fractional-calculus operators, which contains the already-existing operators as particular cases?

This is a natural question to ask from a mathematical point of view: generalisation is always a key concept in mathematics. Armed with a general formalism which includes all the specific fractional-calculus operators, we can then attempt to prove results and establish a theory just for this general model, rather than proving similar results many times in many different models.

A similar question has also been proposed by researchers working in applied sciences, e.g. engineers, from the point of view of real-world applications. They wish, in principle, to have a simple and efficient structure for fractional calculus -- if possible, just one single model, as in the classical case -- which can be used to model many different real-life processes. In our opinion, this fundamental question is still an open problem: currently, experimental data in several different contexts corresponds to several separate models of fractional calculus, but if these can be unified in a single framework, then that generalised framework is all we need. In any case, real data should be taken as top criteria in validating a given fractional model. In this way, \textit{a priori} an equal chance is given to all fractional models but, finally, only one model is chosen as being the most efficient for a specific set of real data.  We are sure that by improving the numerical schemes it will be possible to see sharp differences between the diverse proposed fractional models and criteria for what a fractional operator means. We believe that research in the field of fractional calculus is interdisciplinary and soon we will have improved fractional models in many fields of science and engineering.

One direction of research in this area has been to add more indices and parameters, as we see for example with the multi-parameter Prabhakar model mentioned above. However, although these complex operators are very nice mathematically, the laws of nature are always simple. Therefore, in order to satisfy the concerns of both pure mathematics and applications, we have to find a compromise between the complex forms of generalised fractional operators and the simplicity of the laws of nature.
\end{question}

In the current work, we consider a general framework of operators, which includes many of the proposed models of fractional calculus mentioned above, and whose relevance to fractional calculus can be justified in a clear and objective manner. The extreme generality of our approach enables us to consider many types of fractional operators as special cases, but we are still able to adapt some of the standard tools of fractional calculus to our general operators. Previous researchers such as \cite{atanackovic-pilipovic-zorica,kochubei} have noticed that such a general framework exists, but they did not analyse it in enough detail to grasp its true power. Here, inspired by some existing results on the AB \cite{baleanu-fernandez} and Prabhakar \cite{giusti,fernandez-baleanu-srivastava} models, we demonstrate how our general model can be expressed solely in terms of the classical Riemann--Liouville operators \eqref{RLdef:int} by means of a series formula. This confirms that our definition forms part of the field of fractional calculus, as well as enabling us to prove several theorems about it which are analogous to basic theorems in classical calculus.

The organisation of this paper is as follows. In Section 2.1 we define a new general class of fractional integral operators and establish many of their fundamental properties. In Section 2.2 we consider how to define fractional derivatives in this general model. In Section \ref{sec:transODE} we prove some results on transforms and consider differential equations in the new model. In section \ref{sec:prodchain} we establish Leibniz and chain rules in the new model. In section \ref{sec:CauchyVolterra} we solve a general Cauchy problem in the new model. In section \ref{sec:conclusions} we conclude the paper.

\section{Definition and basic properties} \label{sec:main}

\subsection{Fractional integrals}

We propose the following definition for a new general class of fractional operators.

\begin{definition}
\label{E:defn}
Let $[a,b]$ be a real interval, $\alpha$ and $\beta$ be complex parameters with non-negative real parts, and $R$ be a positive number satisfying $R>(b-a)^{\mathrm{Re}(\beta)}$. Let $A$ be a complex function analytic on the disc $D(0,R)$ and defined on this disc by the locally uniformly convergent power series
\begin{equation}
\label{A:entire}
A(x)=\sum_{n=0}^{\infty}a_nx^n,
\end{equation}
where the coefficients $a_n=a_n(\alpha,\beta)$ are permitted to depend on $\alpha$ and $\beta$ if desired. We define the following fractional integral operator, acting on a function $f:[a,b]\rightarrow\mathbb{R}$ with properties to be determined later (for example, $f\in L^1[a,b]$ -- see Theorem \ref{L1:thm} below):
\begin{align}
\label{EIdef} \prescript{A}{}I_{a+}^{\alpha,\beta}f(t)\coloneqq \int_a^t(t-\tau)^{\alpha-1}A\left((t-\tau)^{\beta}\right)f(\tau)\,\mathrm{d}\tau.
\end{align}
\end{definition}

The formula \eqref{EIdef} is an extreme generalisation of the assortment of fractional models we considered above. In order for this new definition to be useful, we shall need to impose at least some further conditions on the function $A$. But before proceeding to such considerations, let us state formally how the existing fractional models can be seen as cases of this new generalism.

\begin{remark}
\label{E:cases}
It has been demonstrated \cite{giusti,fernandez-baleanu-srivastava} that the Prabhakar kernel is general enough to include several other kernel functions of fractional calculus, including the AB one, as special cases. We now demonstrate that all of the fractional models mentioned above can be viewed as special cases of our new generalised model. For appropriate functions $f$ and parameters $a$, $\alpha$, $\beta$, etc., we have the following correspondences.

The classical integer-order iterated integral is a special case given by:
\begin{equation}
\label{Eclassic:eqn} I^n_{a+}f(t)=\frac{1}{(n-1)!A(1)}\prescript{A}{}I_{a+}^{n,0}f(t),
\end{equation}
for an arbitrary choice of function $A$.

Similarly, the RL integral \eqref{RLdef:int} is a special case given by:
\begin{equation}
\label{ERL:eqn1} \prescript{RL}{}I^{\alpha}_{a+}f(t)=\frac{1}{\Gamma(\alpha)A(1)}\prescript{A}{}I_{a+}^{\alpha,0}f(t),
\end{equation}
for an arbitrary choice of function $A$. However, it is often more useful to choose a specific function $A$, and the following choice seems most natural:
\begin{align}
\label{ERL:fns} A(x)&=\frac{1}{\Gamma(\alpha)}, \\
\label{ERL:eqn} \prescript{RL}{}I^{\alpha}_{a+}f(t)&=\prescript{A}{}I_{a+}^{\alpha,0}f(t).
\end{align}
An alternative expression for the RL integral is as follows:
\begin{align}
\label{ERL2:fns} A(x)&=\frac{1}{\Gamma(\alpha)}x, \\
\label{ERL2:eqn} \prescript{RL}{}I^{\alpha}_{a+}f(t)&=\prescript{A}{}I_{a+}^{1,\alpha-1}f(t).
\end{align}
But this seems less natural than the previous correspondence.


The ABR and ABC derivatives \eqref{ABRdef:deriv}--\eqref{ABCdef:deriv} are special cases given by:
\begin{align}
\label{EAB:fns} A(x)&=\frac{B(\alpha)}{1-\alpha}E_{\alpha}\left(\frac{-\alpha}{1-\alpha}x\right), \\
\label{EABR:eqn} \prescript{ABR}{}D_{a+}^{\alpha}f(t)&=\frac{\mathrm{d}}{\mathrm{d}t}\prescript{A}{}I_{a+}^{1,\alpha}f(t), \\
\label{EABC:eqn} \prescript{ABC}{}D_{a+}^{\alpha}f(t)&=\prescript{A}{}I_{a+}^{1,\alpha}f'(t).
\end{align}

The GPF integral \eqref{GPFdef:int} is a special case given by:
\begin{align}
\label{EGPF:fns} A(x)&=\frac{1}{\rho^{\alpha}\Gamma(\alpha)}\exp\left(\frac{\rho-1}{\rho}x\right), \\
\label{EGPF:eqn} \prescript{GPF}{}I_{a+}^{\alpha,\rho}f(t)&=\prescript{A}{}I_{a+}^{\alpha,1}f(t).
\end{align}

The Prabhakar integral \eqref{PRABdef:int} is a special case given by:
\begin{align}
\label{EPRAB:fns} A(x)&=E^{\rho}_{\beta,\alpha}(\omega x), \\
\label{EPRAB:eqn} \prescript{P}{\beta,\omega}I_{a+}^{\alpha,\rho}f(t)&=\prescript{A}{}I_{a+}^{\alpha,\beta}f(t).
\end{align}

We note that the multiplier functions which are used outside of the integral in the AB and GPF definitions can now be absorbed into the coefficients of the function $A$. This gives a cleaner expression for the fractional integral formula: in the definition \eqref{EIdef}, there is no need for any multiplier function outside of the integral.

It is also important to note that our formulation is not general enough to cover \textit{all} models of fractional calculus which have ever been proposed. The field of fractional calculus is very broad and covers a wide variety of different types of operators.
\end{remark}

\begin{remark}
There are some existing papers in the literature which propose definitions of fractional operators involving very general kernel functions -- see for example Kochubei \cite{kochubei}, Agrawal \cite{agrawal}, and Zhao and Luo \cite{zhao-luo}. Let us now compare these with our Definition \ref{E:defn}.

In each of the previous definitions, a very general kernel function was used, yielding operators of the following essential form:
\begin{equation}
\label{previous}
\int_a^tk(t-\tau,\alpha)f(\tau)\,\mathrm{d}\tau,
\end{equation}
where $a$ is a constant of integration, $\alpha$ is a parameter, and $k$ is a general kernel function satisfying some very mild conditions. The authors of the previous papers studied such operators in various ways: formulating and solving differential equations \cite{kochubei}, analysing variational problems \cite{agrawal}, and modelling heat conduction processes \cite{zhao-luo}. However, they were unable to prove in these models certain other properties which might be expected of fractional differintegrals, such as a Leibniz rule or composition properties of the operators.

This is because their formulation is more general than ours: they expanded the scope of the definition so far that it no longer has a clear connection to fractional calculus. The operation defined by \eqref{previous} is essentially a convolution with the general function $k$. For certain choices of $k$, such a convolution operation does reduce to the well-known fractional differentiation and integration. But in general it is hard to describe the ``fractionality'' of the operators, or to see why the parameter $\alpha$ would represent the order of differentiation or integration.

In our work, we have chosen the level of generality to be slightly lower in the definitions. Instead of taking a completely general kernel function $k$, we choose our kernel function to be a general analytic function of a fractional power. The explicit involvement of fractional power functions allows us to easily define an order for our fractional differintegrals. The assumption of analyticity (as we shall see below) enables us to prove identities which connect our operators firmly back to the classical fractional calculus, which in turn enables us to prove many important results analogous to those in the classical models. Although the previously defined operators such as \eqref{previous} are known to be useful in their own right, we believe that our slightly less general operators are more clearly part of fractional calculus.
\end{remark}

\begin{definition}
\label{AGamma:defn}
For any analytic function $A$ as in Definition \ref{E:defn}, we define $A_{\Gamma}$ to be the transformed function
\begin{equation}
\label{AGamma:eqn}
A_{\Gamma}(x)=\sum_{n=0}^{\infty}a_n\Gamma(\beta n+\alpha)x^n.
\end{equation}
The relationship between the pair of functions $A$ and $A_{\Gamma}$ is vital to the understanding of our generalised operators. Although the operator $\prescript{A}{}I_{a+}^{\alpha,\beta}$ itself is most easily defined using the function $A$, many of the results concerning it are more elegant when written in terms of $A_{\Gamma}$ instead of $A$: for example, see Theorem \ref{Eseries:thm} and Theorem \ref{Laplace:thm}.

By the ratio test and Stirling's approximation, the series \eqref{AGamma:eqn} for $A_{\Gamma}$ has radius of convergence given by \[\lim_{n\rightarrow\infty}\left|\frac{a_{n}}{a_{n+1}}(\beta n+\beta+\alpha)^{-\beta}\right|.\] Also by the ratio test, the radius of convergence of the series \eqref{A:entire} for $A$ is $\lim_{n\rightarrow\infty}\left|\frac{a_{n}}{a_{n+1}}\right|\geq R$. Thus we see that if the series for $A_{\Gamma}$ converges, then so does the series for $A$, but not vice versa.
\end{definition}

\begin{theorem}
\label{L1:thm}
With all notation as in Definition \ref{E:defn}, we have a well-defined bounded operator \[\prescript{A}{}I_{a+}^{\alpha,\beta}:L^1[a,b]\rightarrow L^1[a,b]\] for any fixed $\alpha$ and $\beta$ with $\mathrm{Re}(\alpha),\mathrm{Re}(\beta)\geq0$.
\end{theorem}

\begin{proof}
First we prove that for any function $f\in L^1[a,b]$, the resulting function $\prescript{A}{}I_{a+}^{\alpha,\beta}f$ is also in $L^1[a,b]$. For this it will suffice to show that the definite absolute integral \[\int_a^b\left|\prescript{A}{}I_{a+}^{\alpha,\beta}f(t)\right|\,\mathrm{d}t\leq\int_a^b\int_a^t\left|(t-\tau)^{\alpha-1}A\left((t-\tau)^{\beta}\right)f(\tau)\right|\,\mathrm{d}\tau\,\mathrm{d}t\] is finite. By Fubini's theorem, this is equivalent to showing that \[\int_a^b\int_\tau^b\left|(t-\tau)^{\alpha-1}A\left((t-\tau)^{\beta}\right)f(\tau)\right|\,\mathrm{d}t\,\mathrm{d}\tau\] is finite. The latter double integral can be rearranged to \[\int_a^b|f(\tau)|\int_0^{b-\tau}\left|u^{\alpha-1}A(u^{\beta})\right|\,\mathrm{d}u\,\mathrm{d}\tau\leq\int_a^b|f(\tau)|\,\mathrm{d}\tau\int_0^{b-a}\left|u^{\alpha-1}A(u^{\beta})\right|\,\mathrm{d}u.\] Because $A$ is analytic on $D(0,R)$, the function $A(u^{\beta})$ is bounded on the finite interval $[0,b-a]$. And $f$ is an $L^1$ function, so the double integral is bounded as required.

We have also now shown that \[\left\|\prescript{A}{}I_{a+}^{\alpha,\beta}f\right\|_1\leq\|f\|_1\int_0^{b-a}\left|u^{\alpha-1}A(u^{\beta})\right|\,\mathrm{d}u.\] This proves that $\prescript{A}{}I_{a+}^{\alpha,\beta}$ is a bounded operator on $L^1[a,b]$, with operator norm at most $(b-a)^{\alpha}M$, where \[M\coloneqq\sup_{|x|<(b-a)^{\beta}}|A(x)|.\] Note that we have used the assumptions $\mathrm{Re}(\alpha),\mathrm{Re}(\beta)\geq0$ in order to know that the final integral is well-behaved near $u=0$.
\end{proof}

\begin{theorem}[Series formula]
\label{Eseries:thm}
With all notation as in Definition \ref{E:defn}, for any function $f\in L^1[a,b]$, we have the following locally uniformly convergent series for $\prescript{A}{}I_{a+}^{\alpha,\beta}f$ as a function on $[a,b]$:
\begin{equation}
\label{Eseries:eqn}
\prescript{A}{}I_{a+}^{\alpha,\beta}f(t)=\sum_{n=0}^{\infty}a_n\Gamma(\beta n+\alpha)\prescript{RL}{}I_{a+}^{\alpha+n\beta}f(t).
\end{equation}
Alternatively, this identity can be written more concisely in terms of the transformed function $A_{\Gamma}$ introduced in Definition \ref{AGamma:defn}:
\begin{equation}
\label{Eseries:AGamma}
\prescript{A}{}I_{a+}^{\alpha,\beta}f(t)=A_{\Gamma}\left(\prescript{RL}{}I_{a+}^{\beta}\right)\prescript{RL}{}I_{a+}^{\alpha}f(t).
\end{equation}
\end{theorem}

\begin{proof}
We substitute the definition \eqref{A:entire} into the original formula \eqref{EIdef}:
\begin{align*}
\prescript{A}{}I_{a+}^{\alpha,\beta}f(t)&=\int_a^t(t-\tau)^{\alpha-1}\sum_{n=0}^{\infty}\left[a_n(t-\tau)^{\beta n}\right]f(\tau)\,\mathrm{d}\tau \\
&=\int_a^t\sum_{n=0}^{\infty}a_n(t-\tau)^{\beta n+\alpha-1}f(\tau)\,\mathrm{d}\tau.
\end{align*}
The series here is locally uniformly convergent, since $0\leq\left|(t-\tau)^{\beta}\right|\leq(b-a)^{\mathrm{Re}(\beta)}<R$ and the series \eqref{A:entire} for $A$ is assumed to be locally uniformly convergent on $D(0,R)\subset\mathbb{C}$. So the order of integration and summation can be swapped, to get:
\begin{align*}
\prescript{A}{}I_{a+}^{\alpha,\beta}f(t)&=\sum_{n=0}^{\infty}\int_a^ta_n(t-\tau)^{\beta n+\alpha-1}f(\tau)\,\mathrm{d}\tau \\
&=\sum_{n=0}^{\infty}a_n\Gamma(\beta n+\alpha)\left[\frac{1}{\Gamma(\beta n+\alpha)}\int_a^t(t-\tau)^{\beta n+\alpha-1}f(\tau)\,\mathrm{d}\tau\right].
\end{align*}
By the definition \eqref{RLdef:int} of Riemann--Liouville fractional integrals, the expression in square brackets is precisely the $(\beta n+\alpha)$th RL integral of $f(t)$. And the result follows.
\end{proof}

The above two theorems are fundamental for the study of these general operators. Theorem \ref{L1:thm} establishes a domain of definition for the operator, namely the space of all $L^1$ functions on the interval. This is the same domain of definition that works for the RL \cite{samko-kilbas-marichev}, AB \cite{baleanu-fernandez}, and Prabhakar \cite{prabhakar} operators. Theorem \ref{Eseries:thm} establishes a way of expressing the general operators in terms of only the classical Riemann--Liouville fractional integrals, following the method used in \cite{baleanu-fernandez,fernandez-baleanu-srivastava}. This is the reason for our assumption that $A$ was analytic, i.e. has a convergent power series. It is immensely significant because, as well as cementing the position of these operators as part of fractional calculus, it also provides us with short-cuts to many useful theorems concerning them. Many well-known results on RL integrals can now be quickly extended to a much more general scenario.

\begin{theorem}
With all notation as in Definition \ref{E:defn}, we have a well-defined operator \[\prescript{A}{}I_{a+}^{\alpha,\beta}:C[a,b]\rightarrow C(a,b)\] for any fixed $\alpha$ and $\beta$ with $\mathrm{Re}(\alpha),\mathrm{Re}(\beta)\geq0$.
\end{theorem}

\begin{proof}
First we note that, by Theorem \ref{L1:thm}, the operator $\prescript{A}{}I_{a+}^{\alpha,\beta}$ is well-defined on $C[a,b]$ since this function space is a subset of $L^1[a,b]$.

If $f$ is continuous on the open interval $(a,b)$, then so is its Riemann--Liouville integral $\prescript{RL}{}I_{a+}^{\nu}$ for any $\nu$ with positive real part \cite{samko-kilbas-marichev}. And the series \eqref{Eseries:eqn} is locally uniformly convergent. So we have an expression for $\prescript{A}{}I_{a+}^{\alpha,\beta}f(t)$ as a locally uniformly convergent series of functions in $C(a,b)$, which must itself be in $C(a,b)$ as required.
\end{proof}

\begin{theorem}
\label{RL:E:thm}
Let $a$, $b$, $A$ be as in Definition \ref{E:defn}. For any $f\in L^1[a,b]$ and $\alpha,\beta,\gamma\in\mathbb{C}$ with non-negative real parts, the composition of Riemann--Liouville fractional integrals with our generalised operators is given by:
\begin{align}
\label{RL:E:eqn}
\prescript{RL}{}I_{a+}^{\gamma}\circ\prescript{A}{}I_{a+}^{\alpha,\beta}f(t)=\prescript{A}{}I_{a+}^{\alpha,\beta}\circ\prescript{RL}{}I_{a+}^{\gamma}f(t)&=\prescript{A}{}I_{a+}^{\alpha+\gamma,\beta}f(t) \\
\nonumber &=A_{\Gamma}\left(\prescript{RL}{}I_{a+}^{\beta}\right)\prescript{RL}{}I_{a+}^{\alpha+\gamma}f(t).
\end{align}
\end{theorem}

\begin{proof}
This follows directly from the expressions \eqref{Eseries:eqn} and \eqref{Eseries:AGamma} for $\prescript{A}{}I_{a+}^{\alpha,\beta}f(t)$, using the semigroup property of Riemann--Liouville fractional integrals.
\end{proof}

Theorem \ref{RL:E:thm} is interesting because it may provide us with a way of solving fractional integro-differential equations that involve both Riemann--Liouville operators and our new generalised operators. We investigate such problems in more detail in Section \ref{sec:CauchyVolterra} below. 

\begin{theorem}
\label{commutativity:thm}
Let $a$, $b$, $A$ be as in Definition \ref{E:defn}. The set \[\left\{\prescript{A}{}I_{a+}^{\alpha,\beta}:\alpha,\beta\in\mathbb{C},\mathrm{Re}(\alpha)\geq0,\mathrm{Re}(\beta)\geq0\right\}\] forms a commutative family of operators on the function space $L^1[a,b]$.
\end{theorem}

\begin{proof}
This follows directly from Theorem \ref{Eseries:thm} and the commutativity of Riemann--Liouville fractional integrals. Explicitly, we have:
\begin{align*}
\prescript{A}{}I_{a+}^{\alpha_1,\beta_1}&\circ\prescript{A}{}I_{a+}^{\alpha_2,\beta_2}f(t) \\
&=\sum_{n=0}^{\infty}a_n\Gamma(\beta_1 n+\alpha_1)\prescript{RL}{}I_{a+}^{\alpha_1+n\beta_1}\left[\sum_{m=0}^{\infty}a_m\Gamma(\beta_2 m+\alpha_2)\prescript{RL}{}I_{a+}^{\alpha_2+m\beta_2}f(t)\right] \\
&=\sum_{m,n}a_n\Gamma(\beta_1 n+\alpha_1)a_m\Gamma(\beta_2 m+\alpha_2)\prescript{RL}{}I_{a+}^{\alpha_1+n\beta_1}\circ\prescript{RL}{}I_{a+}^{\alpha_2+m\beta_2}f(t) \\
&=\sum_{m,n}a_m\Gamma(\beta_2 m+\alpha_2)a_n\Gamma(\beta_1 n+\alpha_1)\prescript{RL}{}I_{a+}^{\alpha_2+m\beta_2}\circ\prescript{RL}{}I_{a+}^{\alpha_1+n\beta_1}f(t) \\
&=\prescript{A}{}I_{a+}^{\alpha_2,\beta_2}\circ\prescript{A}{}I_{a+}^{\alpha_1,\beta_1}f(t).
\end{align*}
\end{proof}

\begin{theorem}[Semigroup property in one parameter]
\label{semigroup:thm}
Let $a$, $b$, $A$ be as in Definition \ref{E:defn}, and fix $\alpha_1,\alpha_2,\beta\in\mathbb{C}$ with non-negative real parts. The semigroup property \[\prescript{A}{}I_{a+}^{\alpha_1,\beta}\circ\prescript{A}{}I_{a+}^{\alpha_2,\beta}f(t)=\prescript{A}{}I_{a+}^{\alpha_1+\alpha_2,\beta}f(t)\] is uniformly valid (regardless of $\alpha_1$, $\alpha_2$, $\beta$, and $f$) if and only if the following condition is satisfied for all non-negative integers $k$:
\begin{align}
\label{semigroup:condn} \sum_{m+n=k}a_n(\alpha_1,\beta)a_m(\alpha_2,\beta)B(\beta n+\alpha_1,\beta m+\alpha_2)=a_k(\alpha_1+\alpha_2,\beta).
\end{align}
\end{theorem}

\begin{proof}
We saw in the proof of Theorem \ref{commutativity:thm} that
\begin{align*}
\prescript{A}{}I_{a+}^{\alpha_1,\beta}\circ\prescript{A}{}I_{a+}^{\alpha_2,\beta}f(t)&=\sum_{m,n}a_n\Gamma(\beta n+\alpha_1)a_m\Gamma(\beta m+\alpha_2)\prescript{RL}{}I_{a+}^{\alpha_1+n\beta}\circ\prescript{RL}{}I_{a+}^{\alpha_2+m\beta}f(t) \\
&=\sum_{m,n}a_na_m\Gamma(\beta n+\alpha_1)\Gamma(\beta m+\alpha_2)\prescript{RL}{}I_{a+}^{\alpha_1+\alpha_2+(n+m)\beta}f(t) \\
&=\sum_{k=0}^{\infty}\left[\sum_{m+n=k}a_na_m\Gamma(\beta n+\alpha_1)\Gamma(\beta m+\alpha_2)\right]\prescript{RL}{}I_{a+}^{\alpha_1+\alpha_2+k\beta}f(t).
\end{align*}
Meanwhile, the series formula \eqref{Eseries:eqn} yields \[\prescript{A}{}I_{a+}^{\alpha_1+\alpha_2,\beta}f(t)=\sum_{k=0}^{\infty}a_k\Gamma(\beta k+\alpha_1+\alpha_2)\prescript{RL}{}I_{a+}^{\alpha_1+\alpha_2+k\beta}f(t).\] Clearly these two expressions are always equal if and only if 
\begin{equation}
\label{semigroup:condn2}
\sum_{m+n=k}a_n(\alpha_1,\beta)a_m(\alpha_2,\beta)\Gamma(\beta n+\alpha_1)\Gamma(\beta m+\alpha_2)=a_k(\alpha_1+\alpha_2,\beta)\Gamma(\beta k+\alpha_1+\alpha_2)
\end{equation}
for all $k=0,1,2,\dots$, which is equivalent to \eqref{semigroup:condn} by definition of the beta function.
\end{proof}

\begin{remark}
Let us examine the condition \eqref{semigroup:condn2} with increasing values of $k$.

For $k=0$, the equation becomes \[a_0(\alpha_1,\beta)a_0(\alpha_2,\beta)\Gamma(\alpha_1)\Gamma(\alpha_2)=a_0(\alpha_1+\alpha_2,\beta)\Gamma(\alpha_1+\alpha_2).\] Therefore we need $a_0(\alpha,\beta)=\frac{1}{\Gamma(\alpha)}$ for the equation to be uniformly valid.

For $k=1$, the equation becomes \[a_0(\alpha_1,\beta)a_1(\alpha_2,\beta)\Gamma(\alpha_1)\Gamma(\beta+\alpha_2)+a_1(\alpha_1,\beta)a_0(\alpha_2,\beta)\Gamma(\beta+\alpha_1)\Gamma(\alpha_2)=a_1(\alpha_1+\alpha_2,\beta)\Gamma(\beta+\alpha_1+\alpha_2).\] After substituting the known expression for $a_0$, this becomes \[a_1(\alpha_2,\beta)\Gamma(\beta+\alpha_2)+a_1(\alpha_1,\beta)\Gamma(\beta+\alpha_1)=a_1(\alpha_1+\alpha_2,\beta)\Gamma(\beta+\alpha_1+\alpha_2).\] There are many possibilities for $a_1$ which would satisfy this identity. Note that the trivial solution $a_1=0$, followed by setting $a_2=0$, $a_3=0$, etc. to ensure that \eqref{semigroup:condn2} remains valid for all $k$, would yield precisely the Riemann--Liouville fractional model as specified by \eqref{ERL:fns}.
\end{remark}

\begin{theorem}
\label{semigroup2:thm}
Let $a$, $b$, $A$ be as in Definition \ref{E:defn}, and fix $\alpha_1,\alpha_2,\beta_1,\beta_2\in\mathbb{C}$ with non-negative real parts. The semigroup property \[\prescript{A}{}I_{a+}^{\alpha_1,\beta_1}\circ\prescript{A}{}I_{a+}^{\alpha_2,\beta_2}f(t)=\prescript{A}{}I_{a+}^{\alpha_1+\alpha_2,\beta_1+\beta_2}f(t)\] cannot be uniformly valid for arbitrary $\alpha_1$, $\alpha_2$, $\beta_1$, $\beta_2$, and $f$.
\end{theorem}

\begin{proof}
Again we can use the composition formula found in the proof of Theorem \ref{commutativity:thm}:
\begin{align*}
\prescript{A}{}I_{a+}^{\alpha_1,\beta_1}\circ\prescript{A}{}I_{a+}^{\alpha_2,\beta_2}f(t)&=\sum_{m,n}a_n\Gamma(\beta_1 n+\alpha_1)a_m\Gamma(\beta_2 m+\alpha_2)\prescript{RL}{}I_{a+}^{\alpha_1+\alpha_2+ n\beta_1+m\beta_2}f(t) \\
\begin{split}
&=\sum_{k=0}^{\infty}a_ka_k\Gamma(\beta_1 k+\alpha_1)\Gamma(\beta_2 k+\alpha_2)\prescript{RL}{}I_{a+}^{\alpha_1+\alpha_2+ k(\beta_1+\beta_2)}f(t) \\
&\hspace{2cm}+\sum_{m\neq n}a_n\Gamma(\beta_1 n+\alpha_1)a_m\Gamma(\beta_2 m+\alpha_2)\prescript{RL}{}I_{a+}^{\alpha_1+\alpha_2+ n\beta_1+m\beta_2}f(t).
\end{split}
\end{align*}
Meanwhile, the series formula \eqref{Eseries:eqn} yields \[\prescript{A}{}I_{a+}^{\alpha_1+\alpha_2,\beta_1+\beta_2}f(t)=\sum_{k=0}^{\infty}a_k\Gamma((\beta_1+\beta_2)k+\alpha_1+\alpha_2)\prescript{RL}{}I_{a+}^{\alpha_1+\alpha_2+ k(\beta_1+\beta_2)}f(t).\] For uniform equality, we need to have both \[a_k(\alpha_1,\beta_1)a_k(\alpha_2,\beta_2)\Gamma(\beta_1 k+\alpha_1)\Gamma(\beta_2 k+\alpha_2)=a_k(\alpha_1+\alpha_2,\beta_1+\beta_2)\Gamma((\beta_1+\beta_2)k+\alpha_1+\alpha_2)\] for all $k\in\mathbb{Z}^+_0$ and also \[a_n(\alpha_1,\beta_1)\Gamma(\beta_1 n+\alpha_1)a_m(\alpha_2,\beta_2)\Gamma(\beta_2 m+\alpha_2)=0\] for all distinct $m,n\in\mathbb{Z}^+_0$. But the latter condition implies $A=0$, which makes the whole problem trivial.
\end{proof}

\begin{remark}
The Prabhakar operator does have a semigroup property in two parameters \cite{prabhakar}, but these two parameters are -- in the notation of \eqref{EPRAB:fns} -- $\alpha$ and $\rho$, not $\alpha$ and $\beta$. So the result of Theorem \ref{semigroup2:thm} does not contradict this property of Prabhakar operators.
\end{remark}

\subsection{Fractional derivatives}

In Definition \ref{E:defn} we saw a way to define fractional integrals with general analytic kernel functions. But several of the familiar special cases of this formula, such as the AB fractional derivatives \eqref{EABR:eqn}--\eqref{EABC:eqn}, required taking derivatives as well as applying the integral operator \eqref{EIdef}. This is natural, because defining fractional derivatives in terms of classical derivatives and fractional integrals has been a well-established practice starting from Riemann--Liouville \eqref{RLdef:deriv}. So it now makes sense to ask: given the fractional integral operator $\prescript{A}{}I_{a+}^{\alpha,\beta}$, how might we define a corresponding fractional differential operator?

Clearly we will have operators of both Riemann--Liouville and Caputo type, according to whether we apply the derivative inside or outside the integration. We guess that we should consider operators
\begin{align}
\label{ERdef} \prescript{A}{RL}D_{a+}^{\alpha,\beta}f(t)&=\frac{\mathrm{d}^m}{\mathrm{dt^m}}\left(\prescript{A}{}I_{a+}^{\alpha',\beta'}f(t)\right), \\
\label{ECdef} \prescript{A}{C}D_{a+}^{\alpha,\beta}f(t)&=\prescript{A}{}I_{a+}^{\alpha',\beta'}\left(\frac{\mathrm{d}^m}{\mathrm{dt^m}}f(t)\right),
\end{align}
where the natural number $m$ and the orders $\alpha'$, $\beta'$ depend on $\alpha$ and $\beta$. The question, then, is how these inter-variable dependences are defined. For the RL integral \eqref{ERL:eqn} and the Prabhakar integral \eqref{EPRAB:eqn}, we would use $\beta'=\beta$ and $m+\alpha'=\alpha$. For the AB derivatives \eqref{EABR:eqn} and \eqref{EABC:eqn}, we would use $\alpha'=\alpha$ and $m+\beta'=\beta$. What happens in the general case?

We note the following series formula as a corollary of Theorem \ref{Eseries:thm}.

\begin{corollary}
For any $m\in\mathbb{N}$, with all notation as in Theorem \ref{Eseries:thm}, we have
\begin{align}
\label{Eseries:deriv} \frac{\mathrm{d}^m}{\mathrm{dt^m}}\left(\prescript{A}{}I_{a+}^{\alpha,\beta}f(t)\right)&=\sum_{n=0}^{\infty}a_n\Gamma(\beta n+\alpha)\prescript{RL}{}I_{a+}^{\alpha+n\beta-m}f(t) \\
\label{Eseries:deriv:AGamma}&=A_{\Gamma}\left(\prescript{RL}{}I_{a+}^{\beta}\right)\prescript{RL}{}I_{a+}^{\alpha-m}f(t).
\end{align}
\end{corollary}

\begin{proof}
This follows from \eqref{Eseries:eqn}, using the fact that any classical derivative of a Riemann--Liouville differintegral is another RL differintegral of the appropriate order \cite{kilbas-srivastava-trujillo,miller-ross}.
\end{proof}

Unfortunately, except in a few special cases, we are not able to treat the operator \eqref{Eseries:deriv} as an inverse of the integral operator \eqref{EIdef} with the same function $A$. To see why, consider the result of Theorems \ref{semigroup:thm} and \ref{semigroup2:thm}. We know that a semigroup property in two parameters is impossible, and a semigroup property in one parameter would preserve $\beta$. Therefore, if we want a statement of the form \[\frac{\mathrm{d}^m}{\mathrm{dt^m}}\circ\prescript{A}{}I_{a+}^{\alpha,\beta}\circ\prescript{A}{}I_{a+}^{\alpha',\beta'}f(t)=f(t)\] to be true, then we would need $\beta'=\beta$ and the above equation would reduce to \[\frac{\mathrm{d}^m}{\mathrm{dt^m}}\circ\prescript{A}{}I_{a+}^{\alpha+\alpha',\beta}f(t)=f(t).\] But this requires $\beta=0$, so that the operator on the left-hand side is trivial. And if $\beta=\beta'=0$, then the operators essentially reduce to the Riemann--Liouville fractional differintegrals.

The reason why we \textit{are} able to get a left inverse in the form of \eqref{Eseries:deriv} for the Prabhakar operator \cite{kilbas-saigo-saxena} is that the function $A$ is changed slightly between the fractional integral \eqref{PRABdef:int} and the fractional derivative \eqref{PRABdef:deriv}: namely, by negating the extra parameter $\rho$ as well as changing the parameter $\alpha$.

In general, there are two possible approaches which may be used in our new framework to construct a system of differintegral operators with well-defined derivatives, integrals, and an inversion relation:

\begin{enumerate}
\item \textbf{Approach 1.} We can define the fractional \textbf{integral} by an expression of the form \eqref{EIdef}. Then to define the fractional \textbf{derivative}, we need to find a different analytic function $\bar{A}$ which `complements' the original choice of $A$, in the sense that \[\prescript{\bar{A}}{RL}D_{a+}^{\alpha,\beta}\circ\prescript{A}{}I_{a+}^{\alpha,\beta}f=f.\] In order to find an appropriate $\bar{A}$, we can consider the series formula \eqref{Eseries:AGamma} and invert the function $A_{\Gamma}$.

\item \textbf{Approach 2.} Alternatively, we can define the fractional \textbf{derivative} by an expression of the form \eqref{ERdef} or \eqref{ECdef}. Then to define the fractional \textbf{integral}, we need to find an inverse for the fractional differential operator thus defined. This was the approach used to construct the AB model \cite{atangana-baleanu}.
\end{enumerate}

In order to illustrate the general discussion above, we consider specific implementations of these ideas, which may be used to recover some of the fractional models we already know.

Let us consider the Approach 1 described above, and try to find a condition on $\bar{A}$ in terms of $A$.

We use the notation $\mathcal{I}=\prescript{A}{}I_{a+}^{\alpha,\beta}$ for some fixed $\alpha$, $\beta$, and analytic function $A(x)=\sum_{n=0}^{\infty}a_nx^n$. In other words, we define \[\mathcal{I}f(t)=\int_a^t(t-\tau)^{\alpha-1}A\left((t-\tau)^{\beta}\right)f(\tau)\,\mathrm{d}\tau.\] We also use the notation $\mathcal{D}=\prescript{\bar{A}}{RL}D_{a+}^{\alpha,\beta}$, as defined in \eqref{ERdef}, for some fixed $\alpha'$, $\beta'$, $m$, and analytic function $\bar{A}(x)=\sum_{n=0}^{\infty}\bar{a_n}x^n$. In other words, we define \[\mathcal{D}f(t)=\frac{\mathrm{d}^m}{\mathrm{dt^m}}\left(\int_a^t(t-\tau)^{\alpha'-1}\bar{A}\left((t-\tau)^{\beta'}\right)f(\tau)\,\mathrm{d}\tau\right).\] Let us check the conditions for $\mathcal{D}$ to be a left inverse of $\mathcal{I}$, using the expression \eqref{Eseries:AGamma} for these operators as series in terms of $A_{\Gamma}$:
\begin{align*}
\mathcal{D}\circ\mathcal{I}f(t)=f(t)&\Leftrightarrow\left[\bar{A}_{\Gamma}\left(\prescript{RL}{}I_{a+}^{\beta'}\right)\prescript{RL}{}I_{a+}^{\alpha'-m}\right]\circ\left[A_{\Gamma}\left(\prescript{RL}{}I_{a+}^{\beta}\right)\prescript{RL}{}I_{a+}^{\alpha}\right]f(t)=f(t) \\
&\Leftrightarrow\bar{A}_{\Gamma}\left(\prescript{RL}{}I_{a+}^{\beta'}\right)A_{\Gamma}\left(\prescript{RL}{}I_{a+}^{\beta}\right)\prescript{RL}{}I_{a+}^{\alpha+\alpha'-m}f(t)=f(t).
\end{align*}
So the following conditions will suffice to give us a well-defined left inverse operator to \eqref{EIdef}:
\begin{equation}
\label{inverse:condns}
\alpha'=m-\alpha,\quad\quad\beta'=\beta,\quad\quad\bar{A}_{\Gamma}\cdot A_{\Gamma}=1.
\end{equation}

The above discussion can be formalised into the following definition for generalised fractional differentiation operators.

\begin{definition}
\label{Eder:defn}
Let $a$, $b$, $\alpha$, $\beta$, and $A$ be as in Definition \ref{E:defn}. Using Approach 1 as discussed above, we can define the following fractional differential operators, of both Riemann--Liouville and Caputo type, acting on a function $f:[a,b]\rightarrow\mathbb{R}$ with sufficient differentiability properties.
\begin{align}
\label{Eder:defnRL} \prescript{A}{RL}D_{a+}^{\alpha,\beta}f(t)&=\frac{\mathrm{d}^m}{\mathrm{dt^m}}\left(\prescript{\bar{A}}{}I_{a+}^{m-\alpha,\beta}f(t)\right), \\
\label{Eder:defnC} \prescript{A}{C}D_{a+}^{\alpha,\beta}f(t)&=\prescript{\bar{A}}{}I_{a+}^{m-\alpha,\beta}\left(\frac{\mathrm{d}^m}{\mathrm{dt^m}}f(t)\right),
\end{align}
where the function $\bar{A}$ used on the right-hand side is defined such that $A_{\Gamma}(x)\cdot\bar{A}_{\Gamma}(x)=1$. (Here the $\Gamma$-transformed functions are as defined in Definition \ref{AGamma:defn}.)
\end{definition}

\begin{example}
Let us consider the Prabhakar model. Here the fractional integral is defined using the function $A(x)=E^{\rho}_{\beta,\alpha}(\omega x)$, and the fractional derivative is defined by an expression of the form \eqref{ERdef} using $\alpha'=m-\alpha$, $\beta'=\beta$, and the function $\bar{A}(x)=E^{-\rho}_{\beta,m-\alpha}(\omega x)$. In order to verify that \eqref{inverse:condns} is valid, we just need to check the final part. Here we have:
\begin{align*}
A(x)&=E^{\rho}_{\beta,\alpha}(\omega x)=\sum_{n=0}^{\infty}\frac{\Gamma(\rho+n)\omega^n}{\Gamma(\rho)\Gamma(\beta n+\alpha)n!}x^n; \\
A_{\Gamma}(x)&=\sum_{n=0}^{\infty}\frac{\Gamma(\rho+n)\omega^n}{\Gamma(\rho)n!}x^n; \\
\bar{A}(x)&=E^{-\rho}_{\beta,m-\alpha}(\omega x)=\sum_{n=0}^{\infty}\frac{\Gamma(-\rho+n)\omega^n}{\Gamma(-\rho)\Gamma(\beta n+m-\alpha)n!}x^n; \\
\bar{A}_{\Gamma}(x)&=\sum_{n=0}^{\infty}\frac{\Gamma(-\rho+n)\omega^n}{\Gamma(-\rho)n!}x^n; \\
\bar{A}_{\Gamma}\cdot A_{\Gamma}(x)&=\sum_{n_1=0}^{\infty}\sum_{n_2=0}^{\infty}\frac{\Gamma(-\rho+n_1)\Gamma(\rho+n_2)\omega^{n_1+n_2}}{\Gamma(-\rho)\Gamma(\rho)n_1!n_2!}x^{n_1+n_2}=1,
\end{align*}
where in the last line we have used some basic properties of gamma functions \cite{fernandez-baleanu-srivastava} to get the desired result.
\end{example}

Thus the Approach 1, which was outlined above and formalised in Definition \ref{Eder:defn}, is a valid method for inverting our generalised fractional integral operators, which indeed yields the Prabhakar fractional derivative when we start from the Prabhakar fractional integral.

We now demonstrate how Approach 2 may be used to derive the AB fractional integral when starting from the AB fractional derivative.

\begin{example}
In the 2nd method described above, we start from an operator of the form
\begin{equation}
\label{ex:AB:D1}
\mathcal{D}f(t)=\frac{\mathrm{d}^m}{\mathrm{dt^m}}\left(\prescript{A}{}I_{a+}^{\alpha,\beta}f(t)\right)=A_{\Gamma}\left(\prescript{RL}{}I_{a+}^{\beta}\right)\prescript{RL}{}I_{a+}^{\alpha-m}f(t).
\end{equation}
To find the inverse of this operator, it would help to know the reciprocal of the function $A_{\Gamma}$. As a basic example, let us consider the case where $A_{\Gamma}(x)=(1-x)^{-1}$. This yields the following choices of function:
\begin{align*}
A_{\Gamma}(x)&=(1-x)^{-1}=\sum_{n=0}^{\infty}x^n; \\
A(x)&=\sum_{n=0}^{\infty}\frac{x^n}{\Gamma(\beta n+\alpha)}=E_{\beta,\alpha}(x).
\end{align*}
Setting $\alpha=m=1$ for simplicity, we find that
\begin{align}
\nonumber \mathcal{D}f(t)&=\frac{\mathrm{d}}{\mathrm{d}t}\left(\int_a^t(t-\tau)^{\alpha-1}A\left((t-\tau)^{\beta}\right)f(\tau)\,\mathrm{d}\tau\right) \\
\label{ex:AB:D2} &=\frac{\mathrm{d}}{\mathrm{d}t}\left(\int_a^tE_{\beta,\alpha}\left((t-\tau)^{\beta}\right)f(\tau)\,\mathrm{d}\tau\right).
\end{align}
For the inverse operator, we can see from \eqref{ex:AB:D1} that it should be given by
\begin{equation}
\label{ex:AB:I}
\mathcal{I}f(t)=\left(1-\prescript{RL}{}I_{a+}^{\beta}\right)f(t).
\end{equation}
Now the formulae \eqref{ex:AB:D2} and \eqref{ex:AB:I}, up to some multiplicative constants, are precisely the definitions for the AB fractional derivative of Riemann--Liouville type and the AB fractional integral. Thus we have re-derived the AB model of fractional calculus, including the inversion properties of AB differintegrals, as a special case of our more general methodology.



\end{example}

\begin{remark}
The above discussion illustrates the difference in structure between the AB and Prabhakar models of fractional calculus. Although the fractional derivatives look similar in both, the inversion relation between fractional integrals and derivatives is quite different.
\end{remark}

\section{Transforms and differential equations} \label{sec:transODE}

Fourier and Laplace transforms for our generalised operators could be computed from the definition \eqref{EIdef} using the convolution property. However, to get a usable expression by this means, we would need to know how to transform the function $A$ directly. It is more straightforward to find formulae for the transformed functions using the series formula of Theorem \ref{Eseries:thm}.

\begin{theorem}
\label{Laplace:thm}
Let $a=0$, $b>0$, and $\alpha$, $\beta$, $A$ be as in Definition \ref{E:defn}, and let $f\in L^2[a,b]$ with Laplace transform $\hat{f}$. The function $\prescript{A}{}I_{0+}^{\alpha,\beta}f(t)$ has a Laplace transform given by the following formula:
\begin{equation}
\label{Laplace:eqn}
\widehat{\prescript{A}{}I_{0+}^{\alpha,\beta}f}(s)=s^{-\alpha}A_{\Gamma}(s^{-\beta})\hat{f}(s),
\end{equation}
where the function $A_{\Gamma}$ is given in Definition \ref{AGamma:defn}.
\end{theorem}

\begin{proof}
We start from the series formula \eqref{Eseries:eqn}, recalling the uniform convergence of the series there:
\begin{align*}
\widehat{\prescript{A}{}I_{0+}^{\alpha,\beta}f}(s)=\sum_{n=0}^{\infty}a_n\Gamma(\beta n+\alpha)\widehat{\prescript{RL}{}I_{0+}^{\alpha+n\beta}f}(s).
\end{align*}
The Laplace transforms of Riemann--Liouville integrals are well-known \cite{miller-ross,samko-kilbas-marichev}, and so we get:
\begin{align*}
\widehat{\prescript{A}{}I_{0+}^{\alpha,\beta}f}(s)&=\sum_{n=0}^{\infty}a_n\Gamma(\beta n+\alpha)s^{-\alpha-n\beta}\hat{f}(s) \\
&=s^{-\alpha}\hat{f}(s)\sum_{n=0}^{\infty}a_n\Gamma(\beta n+\alpha)s^{-n\beta},
\end{align*}
as required.
\end{proof}

\begin{theorem}
\label{Fourier:thm}
Let $a=-\infty$, $b\in\mathbb{R}$, and $\alpha$, $\beta$, $A$ be as in Definition \ref{E:defn}, and let $f\in L^2[a,b]$ with Fourier transform $\tilde{f}$. The function $\prescript{A}{}I_{+}^{\alpha,\beta}f(t)$ has a Fourier transform given by the following formula:
\begin{equation}
\label{Fourier:eqn}
\widetilde{\prescript{A}{}I_{+}^{\alpha,\beta}f}(k)=k^{-\alpha}e^{i\alpha\pi/2}A_{\Gamma}(k^{-\beta}e^{i\beta\pi/2})\tilde{f}(k),
\end{equation}
where the function $A_{\Gamma}$ is given in Definition \ref{AGamma:defn}.
\end{theorem}

\begin{proof}
We start from the series formula \eqref{Eseries:eqn}, recalling the uniform convergence of the series there:
\begin{align*}
\widetilde{\prescript{A}{}I_{+}^{\alpha,\beta}f}(k)=\sum_{n=0}^{\infty}a_n\Gamma(\beta n+\alpha)\widetilde{\prescript{RL}{}I_{+}^{\alpha+n\beta}f}(k).
\end{align*}
The Fourier transforms of Riemann--Liouville integrals are well-known \cite{samko-kilbas-marichev}, and so we get:
\begin{align*}
\widetilde{\prescript{A}{}I_{+}^{\alpha,\beta}f}(k)&=\sum_{n=0}^{\infty}a_n\Gamma(\beta n+\alpha)(-ik)^{-\alpha-n\beta}\hat{f}(k) \\
&=(-ik)^{-\alpha}\hat{f}(k)\sum_{n=0}^{\infty}a_n\Gamma(\beta n+\alpha)(-ik)^{-n\beta},
\end{align*}
as required.
\end{proof}

Given the above two theorems, we can now attempt to solve some differintegral equations within the framework of the generalised operators. The following result demonstrates a basic example of how this would work.

\begin{theorem}
Let $a=0$, $b>0$, and $\alpha$, $\beta$, $A$ be as in Definition \ref{E:defn}, and let $c\in\mathbb{R}$ and $g\in L^2[a,b]$. The ordinary fractional integral equation 
\begin{equation}
\label{ODE1}
\prescript{A}{}I_{+}^{\alpha,\beta}f(t)+cf(t)=g(t),\quad\quad f(0)=\frac{g(0)}{c},
\end{equation}
has a unique solution $f\in L^2[a,b]$.
\end{theorem}

\begin{proof}
We apply Laplace transforms to the equation \eqref{ODE1} and use the result of Theorem \ref{Laplace:thm}:
\begin{align*}
\hat{g}(s)&=\widehat{\prescript{A}{}I_{+}^{\alpha,\beta}f}(s)+c\hat{f}(s) \\
&=s^{-\alpha}A_{\Gamma}(s^{-\beta})\hat{f}(s)+c\hat{f}(s).
\end{align*}
So we have an explicit expression for the Laplace transform of $f$, namely: \[\hat{f}(s)=\frac{\hat{g}(s)}{s^{-\alpha}A_{\Gamma}(s^{-\beta})+c}.\] This has a unique inverse Laplace transform, giving a unique solution function $f$. We assume the initial condition which is necessary for the ODE itself to be consistent at $t=0$.
\end{proof}

\section{Leibniz rule and chain rule} \label{sec:prodchain}

Two fundamental results in any first course on calculus are the Leibniz rule (or product rule) and the chain rule. Analogues of these results in fractional calculus are well known in the Riemann--Liouville model \cite{miller-ross,podlubny,osler1} and have also been developed in other contexts such as the AB and Prabhakar models \cite{baleanu-fernandez,fernandez-baleanu-srivastava}. It turns out that our proposed model, despite its generality, is still sufficiently close to the classical fractional models that we can prove such results in this context too.




\begin{theorem}[Generalised Leibniz rule]
\label{Leibniz:thm}
If $f\in C[a,b]$ and $g\in C^{\infty}[a,b]$, then for any $\alpha,\beta\in\mathbb{C}$ with non-negative real parts, we have:
\begin{equation}
\label{Leibniz:eqn}
\prescript{A}{}I^{\alpha,\beta}_{a+}\big(f(t)g(t)\big)=\sum_{m=0}^{\infty}\frac{\mathrm{d}^mg}{\mathrm{d}t^m}(t)\sum_{n=0}^{\infty}a_n\Gamma(\beta n+\alpha)\binom{-\alpha-n\beta}{m}\prescript{RL}{}D^{-\alpha-n\beta-m}_{a+}f(t).
\end{equation}
\end{theorem}

\begin{proof}
Our starting point is the following Leibniz rule analogue for Riemann--Liouville differintegrals, which is proved in \cite[Eq. (2.199)]{podlubny}:
\begin{equation}
\label{Leibniz:RL}
\prescript{RL}{}D^{\alpha}_{a+}\big(f(t)g(t)\big)=\sum_{m=0}^N\binom{\alpha}{m}\prescript{RL}{}D^{\alpha-m}_{a+}f(t)\prescript{RL}{}D^{m}g(t)-R_{N,\alpha}(t),
\end{equation}
where $N\geq\mathrm{Re}(\alpha)+1$ and the remainder term is defined by \[R_{N,\alpha}(t)\coloneqq\frac{1}{N!\Gamma(-\alpha)}\int_a^t(t-\tau)^{-\alpha-1}f(\tau)\int_{\tau}^t\prescript{RL}{}D^{N+1}g(u)(t-u)^N\,\mathrm{d}u\,\mathrm{d}\tau.\] In \cite{podlubny} it is shown that \[\lim_{N\rightarrow\infty}R_{N,\alpha}(t)=0,\] and related results are established in \cite{baleanu-fernandez,fernandez-baleanu-srivastava} to prove the convergence of the relevant infinite series. Here, we put together the Leibniz rule \eqref{Leibniz:RL} with the series formula \eqref{Eseries:eqn} for the generalised operators:
\begin{align*}
\prescript{A}{}I^{\alpha,\beta}_{a+}\big(f(t)g(t)\big)&=\sum_{n=0}^{\infty}a_n\Gamma(\beta n+\alpha)\prescript{RL}{}D_{a+}^{-\alpha-n\beta}\big(f(t)g(t)\big) \\
&=\sum_{n=0}^{\infty}a_n\Gamma(\beta n+\alpha)\left[\sum_{m=0}^N\binom{-\alpha-n\beta}{m}\prescript{RL}{}D^{-\alpha-n\beta-m}_{a+}f(t)\prescript{RL}{}D^{m}g(t)-R_{N,-\alpha-n\beta}(t)\right] \\
\begin{split}
&=\sum_{n=0}^{\infty}\sum_{m=0}^Na_n\Gamma(\beta n+\alpha)\binom{-\alpha-n\beta}{m}\prescript{RL}{}D^{-\alpha-n\beta-m}_{a+}f(t)\prescript{RL}{}D^{m}g(t) \\ &\hspace{7cm}-\sum_{n=0}^{\infty}a_n\Gamma(\beta n+\alpha)R_{N,-\alpha-n\beta}(t).
\end{split}
\end{align*}
We know by Theorem \ref{Eseries:thm} that the summation over $n$ is locally uniformly convergent. So the summations over $m$ and $n$ can be swapped, and it remains to prove that \[\lim_{N\rightarrow\infty}\sum_{n=0}^{\infty}a_n\Gamma(\beta n+\alpha)R_{N,-\alpha-n\beta}(t)=0.\] By a change of variables in the double integral as in \cite[Eq. (2.201)]{podlubny}, we find the following expression for $R$, which was also written in \cite{fernandez-baleanu-srivastava}:
\[R_{N,-\alpha-n\beta}(t)=\frac{(-1)^N(t-a)^{N+\alpha+n\beta+1}}{N!\Gamma(\alpha+n\beta)}\int_0^1\int_0^1f\big(a+p(t-a)\big)\prescript{RL}{}D^{N+1}g\big(a+(p+q-pq)(t-a)\big)\,\mathrm{d}p\,\mathrm{d}q.\] We insert this formula into the series for which we need to find the limit, and note that the gamma functions cancel precisely with each other:
\begin{multline*}
\sum_{n=0}^{\infty}a_n\Gamma(\beta n+\alpha)R_{N,-\alpha-n\beta}(t) \\ =\frac{a_n(-1)^N}{N!}(t-a)^{N+\alpha+1}A\left((t-a)^{\beta}\right)\int_0^1\int_0^1f\big(a+p(t-a)\big)\prescript{RL}{}D^{N+1}g\big(a+(p+q-pq)(t-a)\big)\,\mathrm{d}p\,\mathrm{d}q.
\end{multline*}
And this expression tends to zero as $N\rightarrow\infty$, by the same argument as in \cite{podlubny}. Thus, the proof is complete.
\end{proof}

\begin{example}
To illustrate our results, we apply the above Theorem \ref{Leibniz:thm} to the function $te^{kt}$. Taking $f(t)=e^{kt}$ and $g(t)=t$ and $a=-\infty$ in the general identity \eqref{Leibniz:eqn} yields the following:
\begin{align*}
\prescript{A}{}I^{\alpha,\beta}_{a+}\big(te^{kt}\big)&=\sum_{m=0}^{1}\frac{\mathrm{d}^m}{\mathrm{d}t^m}(t)\sum_{n=0}^{\infty}a_n\Gamma(\beta n+\alpha)\binom{-\alpha-n\beta}{m}\prescript{RL}{}D^{-\alpha-n\beta-m}_{+}(e^{kt}) \\
&=\sum_{m=0}^{1}t^{1-m}\sum_{n=0}^{\infty}a_n\Gamma(\alpha+n\beta)\binom{-\alpha-n\beta}{m}k^{\alpha+n\beta+m}e^{kt} \\
&=t\sum_{n=0}^{\infty}a_n\Gamma(\alpha+n\beta)k^{\alpha+n\beta}e^{kt}+\sum_{n=0}^{\infty}a_n\Gamma(\alpha+n\beta)(-\alpha-n\beta)k^{\alpha+n\beta+1}e^{kt} \\
&=\sum_{n=0}^{\infty}a_n\Gamma(\alpha+n\beta)k^{\alpha+n\beta}e^{kt}\big[t-k(\alpha+n\beta)\big].
\end{align*}
Thus, we have an explicit series expression for the fractional differintegral of the function $te^{kt}$, for any $k\in\mathbb{C}$ and in any model of fractional calculus which fits into the general framework we have established.
\end{example}

\begin{theorem}[Generalised chain rule]
\label{chain:thm}
If $f,g\in C^{\infty}[a,b]$ and $\alpha,\beta\in\mathbb{C}$ with non-negative real parts, then
\begin{equation}
\label{chain:eqn}
\prescript{A}{}I^{\alpha,\beta}_{a+}\big(f(g(t))\big)=\sum_{m=0}^{\infty}\sum_{n=0}^{\infty}\frac{a_n(-1)^mt^{\alpha+n\beta+m}}{m!(\alpha+n\beta+m)}\sum_{r=1}^m\frac{\mathrm{d}^rf(g(t))}{\mathrm{d}g(t)^r}\sum_{(r_1,\dots,r_m)}\Bigg[\prod_{j=1}^m\tfrac{j}{r_j!(j!)^{r_j}}\Big(\tfrac{\mathrm{d}^jg(t)}{\mathrm{d}x^j}\Big)^{r_j}\Bigg],
\end{equation}
where for any fixed $m$ and $r$ with $1\leq r\leq m$, the summation over $(r_1,\dots,r_m)$ is over all such $m$-tuples which satisfy $\sum_jr_j=r$ and $\sum_jjr_j=m$.
\end{theorem}

\begin{proof}
We use the result of the previous theorem. Substituting $\mathbbm{1}(t)=1$ instead of $f(t)$ and $f(g(t))$ instead of $g(t)$ in \eqref{Leibniz:eqn}, we find:
\begin{align*}
\prescript{A}{}I^{\alpha,\beta}_{a+}\big(f(g(t))\big)&=\sum_{m=0}^{\infty}\frac{\mathrm{d}^m}{\mathrm{d}t^m}f(g(t))\sum_{n=0}^{\infty}a_n\Gamma(\beta n+\alpha)\binom{-\alpha-n\beta}{m}\prescript{RL}{}D^{-\alpha-n\beta-m}_{a+}\mathbbm{1}(t) \\
&=\sum_{m=0}^{\infty}\frac{\mathrm{d}^m}{\mathrm{d}t^m}f(g(t))\sum_{n=0}^{\infty}a_n\Gamma(\beta n+\alpha)\binom{-\alpha-n\beta}{m}\frac{t^{\alpha+n\beta+m}}{\Gamma(\beta n+\alpha+m+1)} \\
&=\sum_{m=0}^{\infty}\frac{\mathrm{d}^m}{\mathrm{d}t^m}f(g(t))\sum_{n=0}^{\infty}a_n\Gamma(\beta n+\alpha)\frac{\Gamma(1-\alpha-n\beta)}{m!\Gamma(1-\alpha-n\beta-m)}\cdot\frac{t^{\alpha+n\beta+m}}{\Gamma(\beta n+\alpha+m+1)} \\
&=\sum_{m=0}^{\infty}\frac{\mathrm{d}^m}{\mathrm{d}t^m}f(g(t))\sum_{n=0}^{\infty}\frac{a_n\sin(\pi(\alpha+n\beta+m))t^{\alpha+n\beta+m}}{m!(\alpha+n\beta+m)\sin(\pi(\alpha+n\beta))} \\
&=\sum_{m=0}^{\infty}\frac{\mathrm{d}^m}{\mathrm{d}t^m}f(g(t))\sum_{n=0}^{\infty}\frac{a_n(-1)^mt^{\alpha+n\beta+m}}{m!(\alpha+n\beta+m)},
\end{align*}
where we have used the reflection formula $\Gamma(z)\Gamma(1-z)=\frac{\pi}{\sin\pi z}$ for the gamma function. We note that once again all the gamma functions cancel out precisely with each other.

Now the classical Fa\`a di Bruno formula can be applied to the function $\frac{\mathrm{d}^m}{\mathrm{d}t^m}f(g(t))$: we know that
\begin{equation}
\label{FaadiBruno}
\frac{\mathrm{d}^m}{\mathrm{d}t^m}f(g(t))=\sum_{r=1}^m\frac{\mathrm{d}^rf(g(t))}{\mathrm{d}g(t)^r}\sum_{(r_1,\dots,r_m)}\Bigg[\prod_{j=1}^m\tfrac{j}{r_j!(j!)^{r_j}}\Big(\tfrac{\mathrm{d}^jg(t)}{\mathrm{d}x^j}\Big)^{r_j}\Bigg]
\end{equation}
where the summation over $(r_1,\dots,r_m)$ is over the set \[\Big\{(r_1,\dots,r_m)\in\left(\mathbb{Z}^+_0\right)^m:\sum_jr_j=r,\sum_jjr_j=m\Big\}.\] Now the result follows by substituting \eqref{FaadiBruno} into our expression for $\prescript{A}{}I^{\alpha,\beta}_{a+}\big(f(g(t))\big)$.
\end{proof}

Theorems \ref{Leibniz:thm} and \ref{chain:thm} can be used to compute the application of our generalised operators to a wide range of functions which can be generated from elementary ones (e.g. power and exponential functions) by multiplication and composition.

\section{The solution of a Cauchy problem using Volterra integral equations} \label{sec:CauchyVolterra}

In this section, we shall consider a generalised ordinary differintegral equation of the following form:
\begin{equation}
\label{Cauchy:eqn}
\prescript{RL}{}D_{a+}^{\gamma}u(t)=\prescript{A}{}I_{a+}^{\alpha,\beta}f(t,u(t)),
\end{equation}
with some appropriate initial conditions to be specified later. Note that the expression $f(t,u(t))$ is a function of $t$, and therefore our generalised operator can be applied in \eqref{Cauchy:eqn} with respect to $t$.

We first state the following equivalence between the Cauchy problem defined by the differintegral equation \eqref{Cauchy:eqn} and a Volterra integral equation, the latter of which we shall then proceed to solve. 

\begin{lemma}
\label{equiv}
Let $a,b,A$ be as in Definition \ref{E:defn}, and fix $\alpha,\beta,\gamma\in\mathbb{C}$ with non-negative real parts. Define $n=\lceil\gamma\rceil$ and let $C_1,\dots,C_n$ be complex constants. Assume the functions $u:[a,b]\rightarrow\mathbb{R}$ and $f:[a,b]\times\mathbb{R}\rightarrow\mathbb{R}$ are such that $u(t)$ and $f(t,u(t))$ are both in $L^1[a,b]$. Then solving the Cauchy-type problem
\begin{align}
\label{equiv:Cauchy1}
\prescript{RL}{}D_{a+}^{\gamma}u(t)=\prescript{A}{}I_{a+}^{\alpha,\beta}f(t,u(t)),\quad\quad t\in[a,b]; \\
\label{equiv:Cauchy2}
\lim_{t\rightarrow a^+}\left(\prescript{RL}{}D_{a+}^{\gamma-k}u(t)\right)=C_k,\quad\quad k=1,2,\dots,n
\end{align}
for $u\in L^1[a,b]$ is precisely equivalent to solving the Volterra integral equation
\begin{equation}
\label{equiv:Volterra}
u(t)=\sum_{k=1}^n\frac{C_k(t-a)^{\gamma-k}}{\Gamma(\gamma-k+1)}+\prescript{A}{}I_{a+}^{\alpha+\gamma,\beta}f(t,u(t)),\quad\quad t\in[a,b].
\end{equation}
\end{lemma}

\begin{proof}
The underlying fact here is the result of Theorem 1 in \cite{kilbas-bonilla-trujillo}, which was also used to solve a similar but more specific problem in \cite[Lemma 4]{kilbas-saigo-saxena2}.

We first use our Theorem \ref{L1:thm} to note that, since $f(t,u(t))$ is an $L^1$ function by assumption, so too is the right-hand side of the equation \ref{equiv:Cauchy1}. Then by \cite[Theorem 1]{kilbas-bonilla-trujillo}, the solution of the Cauchy problem given by \eqref{equiv:Cauchy1} and \eqref{equiv:Cauchy2} is precisely equivalent to the solution of the Volterra equation
\[u(t)=\sum_{k=1}^n\frac{C_k(t-a)^{\gamma-k}}{\Gamma(\gamma-k+1)}+\frac{1}{\Gamma(\gamma)}\int_a^t(t-\tau)^{\gamma-1}\prescript{A}{}I_{a+}^{\alpha,\beta}f(\tau,u(\tau))\,\mathrm{d}\tau.\]
By our Theorem \ref{RL:E:thm}, the result follows.
\end{proof}

Using the equivalence of the Volterra integral equation, we can now prove that the original Cauchy problem has a unique solution, as follows.

\begin{theorem}
\label{Volterra:thm}
With all notation as in Lemma \ref{equiv}, and assuming that the function $f$ satisfies the following Lipschitz condition in the second variable:
\[|f(x,y_1)-f(x,y_2)|<C|y_1-y_2|,\]
the Volterra integral equation \eqref{equiv:Volterra} has a unique solution $y\in L^1[a,b]$.
\end{theorem}

\begin{proof}
We first consider the restriction of \eqref{equiv:Volterra} to an interval $[a,t_1]\subset[a,b]$, where $t_1\in(a,b)$ is chosen close enough to $a$ such that
\begin{equation}
\label{t1:defn}
C(t_1-a)^{\alpha}\sup_{|x|<(t_1-a)^{\beta}}|A(x)|<1.
\end{equation}
Defining
\begin{equation}
\label{y0:defn}
u_0(t)\coloneqq\sum_{k=1}^n\frac{C_k(t-a)^{\gamma-k}}{\Gamma(\gamma-k+1)},
\end{equation}
we can rewrite the Volterra equation \eqref{equiv:Volterra} as \[u(t)=Tu(t),\] where the function operator $T$ is defined by
\begin{equation}
\label{T:defn}
Tu(t)=u_0(t)+\prescript{A}{}I_{a+}^{\alpha+\gamma,\beta}f(t,u(t))
\end{equation}
We aim to apply the contraction mapping theorem to the operator $T$ acting on the complete metric space $L^1[a,t_1]$. In order for this theorem to be applicable, the proofs of the following two statements will be required.

\begin{enumerate}
\item If $u\in L^1[a,t_1]$, then $Tu\in L^1[a,t_1]$.
\item For any $u_1,u_2\in L^1[a,t_1]$, we have \[\|Tu_1-Tu_2\|_1\leq r\|u_1-u_2\|_1,\] where $r\in(0,1)$ is constant and $\|\cdot\|_1$ denotes the $L^1$ norm on $L^1[a,t_1]$.
\end{enumerate}

\textbf{Proof of statement 1.} We have assumed the function $f$ is such that $f(t,u(t))$ is an $L^1$ function for any $L^1$ function $u$. Therefore, by Theorem \ref{L1:thm}, the right-hand term in \eqref{T:defn} is also $L^1$. And clearly $u_0$ is an $L^1$ function, so the result follows.

\textbf{Proof of statement 2.} By the definition \eqref{T:defn}, we have \[Tu_1-Tu_2=\prescript{A}{}I_{a+}^{\alpha+\gamma,\beta}f(t,u_1(t))-\prescript{A}{}I_{a+}^{\alpha+\gamma,\beta}f(t,u_2(t)).\] Thus we have the following inequalities for the norm in $L^1[a,t_1]$:
\begin{align*}
\left\|Tu_1-Tu_2\right\|_1&=\left\|\prescript{A}{}I_{a+}^{\alpha+\gamma,\beta}\left(f(t,u_1(t))-f(t,u_2(t))\right)\right\|_1 \\
&\leq\left[(t_1-a)^{\alpha}\sup_{|x|<(t_1-a)^{\beta}}|A(x)|\right]\left\|f(t,u_1(t))-f(t,u_2(t))\right\|_1 \\
&\leq\left[(t_1-a)^{\alpha}\sup_{|x|<(t_1-a)^{\beta}}|A(x)|\right]C\left\|u_1-u_2\right\|_1,
\end{align*}
where in the second line we used the proof of Theorem \ref{L1:thm} above, and in the third line we used the assumed Lipschitz condition on $f$. And the constant
\begin{equation}
\label{r:defn}
r\coloneqq C(t_1-a)^{\alpha}\sup_{|x|<(t_1-a)^{\beta}}|A(x)|
\end{equation}
is strictly between 0 and 1 by assumption, so the result follows.

By the contraction mapping theorem, we can now say that the Volterra equation \eqref{equiv:Volterra} has a unique solution $u^*\in L^1[a,t_1]$ defined on the interval $[a,t_1]$.

Now \eqref{equiv:Volterra} can be rewritten as
\begin{equation}
\label{Volterra2}
u(t)=\sum_{k=1}^n\frac{C_k(t-a)^{\gamma-k}}{\Gamma(\gamma-k+1)}+\int_a^{t_1}(t-\tau)^{\alpha-1}A\left((t-\tau)^{\beta}\right)f(\tau)\,\mathrm{d}\tau+\prescript{A}{}I_{t_1+}^{\alpha+\gamma,\beta}f(t,u(t)),
\end{equation}
or equivalently as \[u(t)=T_1u(t),\] where the function operator $T_1$ is defined by \[T_1u(t)=u_{01}(t)+\prescript{A}{}I_{t_1+}^{\alpha+\gamma,\beta}f(t,u(t))\] and the function $u_{01}$ is defined by \[u_{01}(t)\coloneqq\sum_{k=1}^n\frac{C_k(t-a)^{\gamma-k}}{\Gamma(\gamma-k+1)}+\int_a^{t_1}(t-\tau)^{\alpha-1}A\left((t-\tau)^{\beta}\right)f(\tau)\,\mathrm{d}\tau.\] Note that $u_{01}$ is a fixed function, not depending on $u$, since we have already proved that $u$ is uniquely determined on $[a,t_1]$. Thus we can use exactly the same approach as before to prove that the Volterra equation \eqref{Volterra2} has a unique solution $u^*\in L^1[t_1,t_2]$ defined on the interval $[t_1,t_2]$, where $t_2$ is defined (analogously to \eqref{t1:defn}) by requiring the inequality
\begin{equation}
\label{t2:defn}
C(t_2-t_1)^{\alpha}\sup_{|x|<(t_2-t_1)^{\beta}}|A(x)|<1.
\end{equation}
Note that, by comparison of \eqref{t1:defn} and \eqref{t2:defn}, we can define $t_2-t_1=t_1-a$ provided that this yields a value $t_2$ which is still in the interval $[a,b]$.

This argument can be extended indefinitely: each time we find a unique $L^1$ solution on $[t_{i-1},t_i]$, we can then define $t_{i+1}$ using an inequality analogous to \eqref{t2:defn} and find a unique $L^1$ solution on $[t_i,t_{i+1}]$ using the same argument. Since the difference $t_i-t_{i-1}$ can be taken as constant, the process must eventually end when the end of the interval $[a,b]$ is reached. Putting all of the $L^1[t_{i-1},t_i]$ solutions together yields a piecewise defined function $y$ on $[a,b]$ which is the unique solution in $L^1[a,b]$ of the original problem.
\end{proof}

\begin{corollary}
\label{Cauchy:thm}
Let $a,b,A$ be as in Definition \ref{E:defn}, and fix $\alpha,\beta,\gamma\in\mathbb{C}$ with non-negative real parts. Define $n=\lceil\gamma\rceil$ and let $C_1,\dots,C_n$ be complex constants. Assume the functions $u:[a,b]\rightarrow\mathbb{R}$ and $f:[a,b]\times\mathbb{R}\rightarrow\mathbb{R}$ are such that $u(t)$ and $f(t,u(t))$ are both in $L^1[a,b]$, and that $f$ satisfies the following Lipschitz condition in the second variable:
\[|f(x,y_1)-f(x,y_2)|<C|y_1-y_2|.\]
Then the Cauchy problem defined by \eqref{equiv:Cauchy1} and \eqref{equiv:Cauchy2} has a unique solution $u\in L^1[a,b]$.
\end{corollary}

\begin{proof}
By Lemma \ref{equiv}, solving the Cauchy problem \eqref{equiv:Cauchy1}--\eqref{equiv:Cauchy2} is equivalent to solving the Volterra integral equation \eqref{equiv:Volterra}. By Theorem \ref{Volterra:thm}, this Volterra equation has a unique solution $u\in L^1[a,b]$.
\end{proof}

\section{Operators with respect to functions}

A well-known extension of the usual calculus is given by differentiating or integrating a function $f(t)$ with respect to another function $g(t)$ instead of with respect to $t$. For integrals, this is called Riemann--Stieltjes integration. The same concept has been extended to fractional derivatives and integrals \cite{osler1,oldham-spanier,kilbas-srivastava-trujillo,samko-kilbas-marichev}, where it is often known as $\psi$-fractional calculus, due to the notation $\psi(t)$ being used instead of $g(t)$. Detailed studies of this idea and its extensions have been made in recent years by authors including \cite{almeida,sousa-oliveira,almeida-malinowska-monteiro}, but the essential definition is as follows for $\psi$-fractional integration and differentiation in the Riemann--Liouville model:
\begin{alignat}{2}
\label{psi:RLint} \prescript{RL}{\psi(t)}I^{\alpha}_{a+}f(t)&\coloneqq\frac{1}{\Gamma(\alpha)}\int_a^t\psi'(\tau)\big(\psi(t)-\psi(\tau)\big)^{\alpha-1}f(\tau)\,\mathrm{d}\tau,\quad\quad&&\mathrm{Re}(\alpha)>0; \\
\label{psi:RLder} \prescript{RL}{\psi(t)}D^{\alpha}_{a+}f(t)&\coloneqq\left(\frac{1}{\psi'(t)}\cdot\frac{\mathrm{d}}{\mathrm{d}t}\right)^m\left(\prescript{RL}{\psi(t)}I^{m-\alpha}_{a+}f(t)\right),\quad\quad&&\mathrm{Re}(\alpha)\geq0,m\coloneqq\lfloor\mathrm{Re}(\alpha)\rfloor+1.
\end{alignat}
We note that \eqref{psi:RLint}, like \eqref{EIdef}, involves multiplying the function $f(t)$ by an expression containing an arbitrary function ($\psi$ versus $A$). But this is the only similarity between the two types of operators: \eqref{EIdef} is a convolution-type operator with very different structure and behaviour from \eqref{psi:RLint}. The definitions \eqref{psi:RLint}--\eqref{psi:RLder} have been extended to $\psi$-fractional differentiation of Caputo \cite{almeida,almeida-malinowska-monteiro} and Hilfer \cite{sousa-oliveira} type, as well as to other models of fractional calculus such as Atangana--Baleanu and Prabhakar \cite{fernandez-baleanu-icfda}. In a similar, natural, way it is possible to extend the definition to our new generalised model of fractional calculus.

\begin{definition}
\label{psi:E:defn}
Let $[a,b]$ be a real interval, and let $\alpha$, $\beta$, $A$ be as in Definition \ref{E:defn}. For any functions $f$ and $\psi$ defined on $[a,b]$ such that $f$ is an $L^1$ function and $\psi$ is both monotonic and $C^1$, we define the generalised fractional integral of $f(t)$ with respect to $\psi(t)$ as follows:
\begin{equation}
\label{psi:Eint}
\prescript{A}{\psi(t)}I^{\alpha,\beta}_{a+}f(t)\coloneqq\int_a^t\psi'(\tau)\left(\psi(t)-\psi(\tau)\right)^{\alpha-1}A\left(\left(\psi(t)-\psi(\tau)\right)^\beta\right)f(\tau)\,\mathrm{d}\tau.
\end{equation}
\end{definition}

Definition \ref{psi:E:defn} provides a natural extension and combination of two different ways of generalising fractional calculus -- namely, differintegration with respect to functions and differintegration using generalised kernel functions. Additionally, this new formalism enables us to consider as special cases several other classical models of fractional calculus, such as the Hadamard and Erdelyi--Kober fractional differintegrals.

\begin{example}
Using $\psi(t)=\log t$ in Definition \ref{psi:E:defn} enables us to recover a generalised version of the \textbf{Hadamard} model of fractional calculus:
\begin{align*}
\prescript{A}{\log(t)}I^{\alpha,\beta}_{a+}f(t)&=\int_a^t\frac{1}{\tau}\left(\log(t)-\log(\tau)\right)^{\alpha-1}A\left(\left(\log(t)-\log(\tau)\right)^\beta\right)f(\tau)\,\mathrm{d}\tau \\
&=\int_a^t\frac{1}{\tau}\left(\log\left(\tfrac{t}{\tau}\right)\right)^{\alpha-1}A\left(\left(\log\left(\tfrac{t}{\tau}\right)\right)^\beta\right)f(\tau)\,\mathrm{d}\tau.
\end{align*}
If we now set $A(x)=\frac{1}{\Gamma(\alpha)}$ and $\beta=0$ as in \eqref{ERL:fns}, then we recover from this the standard Hadamard fractional integral:
\[\prescript{H}{}I^{\alpha}_{a+}f(t)=\frac{1}{\Gamma(\alpha)}\int_a^t\frac{1}{\tau}\left(\log\left(\tfrac{t}{\tau}\right)\right)^{\alpha-1}f(\tau)\,\mathrm{d}\tau.\]
\end{example}


\begin{example}
Using $\psi(t)=t^{\rho+1}$ in Definition \ref{psi:E:defn} enables us to recover a generalised version of the \textbf{Katugampola} model of fractional calculus:
\begin{align*}
\prescript{A}{t^{\rho+1}}I^{\alpha,\beta}_{a+}f(t)&=\int_a^t(\rho+1)\tau^{\rho}\left(t^{\rho+1}-\tau^{\rho+1}\right)^{\alpha-1}A\left(\left(t^{\rho+1}-\tau^{\rho+1}\right)^\beta\right)f(\tau)\,\mathrm{d}\tau \\
&=(\rho+1)\int_a^t\left(t^{\rho+1}-\tau^{\rho+1}\right)^{\alpha-1}A\left(\left(t^{\rho+1}-\tau^{\rho+1}\right)^\beta\right)\tau^{\rho}f(\tau)\,\mathrm{d}\tau.
\end{align*}
If we now set $A(x)=\frac{(\rho+1)^{-\alpha}}{\Gamma(\alpha)}$ and $\beta=0$ as in \eqref{ERL:fns}, then we recover from this the standard Katugampola fractional integral introduced in \cite{katugampola}:
\[\prescript{K}{}I^{\alpha}_{a+}f(t)=\frac{(\rho+1)^{1-\alpha}}{\Gamma(\alpha)}\int_a^t\left(t^{\rho+1}-\tau^{\rho+1}\right)^{\alpha-1}\tau^{\rho}f(\tau)\,\mathrm{d}\tau.\]
\end{example}

\begin{example}
Using $\psi(t)=t^{\sigma}$ and replacing $f(t)$ by $t^{\sigma\eta}f(t)$ in Definition \ref{psi:E:defn} enables us to recover one possible generalisation of the \textbf{Erdelyi--Kober} model of fractional calculus:
\begin{align*}
\prescript{A}{t^{\sigma}}I^{\alpha,\beta}_{a+}\left(t^{\sigma\eta}f(t)\right)&=\int_a^t\sigma\tau^{\sigma-1}\left(t^{\sigma}-\tau^{\sigma}\right)^{\alpha-1}A\left(\left(t^{\sigma}-\tau^{\sigma}\right)^\beta\right)\tau^{\sigma\eta}f(\tau)\,\mathrm{d}\tau \\
&=\sigma\int_a^t\left(t^{\sigma}-\tau^{\sigma}\right)^{\alpha-1}A\left(\left(t^{\sigma}-\tau^{\sigma}\right)^\beta\right)\tau^{\sigma\eta+\sigma-1}f(\tau)\,\mathrm{d}\tau.
\end{align*}
If we now set $A(x)=\frac{1}{\Gamma(\alpha)}$ and $\beta=0$ as in \eqref{ERL:fns}, and also multiply by $t^{-\sigma(\alpha+\eta)}$, then we recover the Erdelyi--Kober fractional integral as defined in \cite[\S18]{samko-kilbas-marichev}:
\[t^{-\sigma(\alpha+\eta)}\left[\prescript{\frac{1}{\Gamma(\alpha)}}{t^{\sigma}}I^{\alpha,\beta}_{a+}\left(t^{\sigma\eta}f(t)\right)\right]=\frac{\sigma t^{-\sigma(\alpha+\eta)}}{\Gamma(\alpha)}\int_a^t\left(t^{\sigma}-\tau^{\sigma}\right)^{\alpha-1}\tau^{\sigma\eta+\sigma-1}f(\tau)\,\mathrm{d}\tau.\]
\end{example}

The above examples demonstrate that our model with general analytic kernels can be extended to cover even more of the classical models of fractional calculus: not only the Riemann--Liouville model and related formulae with different kernels, but also the Hadamard and Katugampola models and their generalisations.

Although this is not the main concern of the current paper, many of the results we have already proved above for the generalised operator $\prescript{A}{}I_{a+}^{\alpha,\beta}f(t)$ with respect to the variable $t$ can also be proved in a very similar way for the operator $\prescript{A}{\psi(t)}I_{a+}^{\alpha,\beta}f(t)$ with respect to a function $\psi(t)$. For example, we have the following generalised Leibniz rule as an extension of Theorem \ref{Leibniz:thm}.

\begin{theorem}
\label{psi:Leibniz:thm}
If $f\in C[a,b]$ and $g\in C^{\infty}[a,b]$, and if $\psi\in C^1[a,b]$ is monotonic, then for any $\alpha,\beta\in\mathbb{C}$ with non-negative real parts, we have:
\begin{equation}
\label{psi:Leibniz:eqn}
\prescript{A}{\psi(t)}I^{\alpha,\beta}_{a+}\big(f(t)g(t)\big)=\sum_{m=0}^{\infty}\left(\frac{1}{\psi'(t)}\cdot\frac{\mathrm{d}}{\mathrm{d}t}\right)^m\big(g(t)\big)\sum_{n=0}^{\infty}a_n\Gamma(\beta n+\alpha)\binom{-\alpha-n\beta}{m}\prescript{RL}{\psi(t)}I^{\alpha+n\beta+m}_{a+}f(t).
\end{equation}
\end{theorem}

\begin{proof}
The proof is exactly analagous to that of Theorem \ref{Leibniz:thm}, with all derivatives and integrals replaced by their $\psi$-differintegral equivalents.
\end{proof}

A fruitful future direction of research might be to examine how much of the usual fractional calculus can be extended to the operators with general analytic kernels. Here we have indicated one direction of generalisation -- combining the ideas of differintegration with respect to functions and the new class of kernel functions -- but other directions may also be possible in the future.

\section{Conclusions} \label{sec:conclusions}

In this work, we have introduced a new general framework for fractional calculus, which incorporates many existing definitions of fractional integrals and derivatives as special cases. We started by defining an integral operator with a general analytic kernel, and proved that it can be written as an infinite series of Riemann--Liouville integrals. Series appear naturally in fractional calculus, and this result is an indication of the fundamental role of Riemann--Liouville in fractional calculus. After proving some fundamental properties of our generalised operator, we considered how it might be used to define fractional derivatives as well as integrals. We demonstrated how the Leibniz rule and chain rule can be extended to the new operators. Using the method of Fourier and Laplace transforms, we analysed and solved some simple ordinary differential equations in the new general framework. Using the contraction mapping theorem and a Volterra integral equation, we also proved existence and uniqueness for Cauchy problems in a more general class of differential equations. Finally, we indicated a new direction of research for the future by proposing a definition of differintegration with respect to functions in the generalised model.

We believe that this new model can now be used as a way of proving various useful results in a more general context. For example, theorems such as the fractional product rule, chain rule, and Taylor's theorem have been proved in some recently suggested models using series \cite{baleanu-fernandez,fernandez-baleanu-srivastava,fernandez-baleanu2}, and similar methods may now be applicable to a much broader class of fractional operators. In the future, theorems like these and many others could potentially be proved in the generalised fractional calculus which we have introduced here. The advantage of such general results is that they advance the field as a whole, not just in one model or another of fractional calculus but increasing knowledge about many different models simultaneously. Generalisation is one of the most powerful tools in the mathematician's arsenal, and by generalising many fractional models in a single framework, we have opened the possibility for a unified approach to solve many problems in the future.

\section*{Acknowledgements}

The authors would like to thank the editor-in-chief, the editor, and the anonymous reviewers for their useful feedback and suggestions.

\section*{Declaration of interests}

The authors declare that they have no competing interests.

\end{document}